\documentclass[final,12pt]{elsarticle}
\usepackage{xcolor,soul,cancel}
\usepackage[color]{showkeys}
\definecolor{refkey}{rgb}{0,1,1}
\definecolor{labelkey}{rgb}{1,0,0}
\usepackage{amssymb}
\usepackage{amsmath}
\usepackage{amsthm}
\usepackage{graphicx}
\usepackage{float}
\journal{arXiv}

\newtheorem{thm}{Theorem}
\newtheorem{lem}{Lemma}
\newtheorem{cor}{Corollary}
\newtheorem{prop}[thm]{Proposition}

\numberwithin{equation}{section}
\newcommand{\eq} [1] {\begin{equation}\label{#1}\quad}
\newcommand{\en} {\end{equation}}
\newcommand{\scal}[1]{\langle#1\rangle}
\newcommand{\norm}[1]{\left\Vert#1\right\Vert}
\newcommand{\abs}[1]{\left\vert#1\right\vert}

\newcommand{\C}{\mathbb C}

\newcommand{\R}{\mathbb R}

\newcommand{\diag}{\operatorname{diag}}

\newcommand{\im}{\operatorname{Im}}

\newcommand{\conv}{\operatorname{conv}}
\newcommand{\re}{\operatorname{Re}}
\newcommand{\tr}{\operatorname{Tr}}

\newcommand{\Span}{\operatorname{Span}}
\newcommand{\Rk}{\operatorname{Rank}}
\begin{document}

\begin{frontmatter}

\title{On some 4-by-4 matrices \\ with bi-elliptical numerical ranges.\tnoteref{support}}


\author[leiden]{Titas Geryba} 
\ead{t.geryba@umail.leidenuniv.nl, tg1404@nyu.edu}
\address[leiden]{Mathematisch Instituut, Universiteit Leiden\\ Niels Bohrweg 1, 2333 CA Leiden,
Netherlands}
\author[nyuad]{Ilya M. Spitkovsky}
\address[nyuad]{Division of Science and Mathematics, New York  University Abu Dhabi (NYUAD), Saadiyat Island,
P.O. Box 129188 Abu Dhabi, United Arab Emirates}
\ead{imspitkovsky@gmail.com, ims2@nyu.edu, ilya@math.wm.edu}
\tnotetext[support]{The second author [IMS]  was supported in part by Faculty Research funding from the Division of Science and Mathematics, New York University Abu Dhabi.}

\begin{abstract}A complete description of 4-by-4 matrices $\begin{bmatrix}\alpha I & C \\D & \beta I\end{bmatrix}$, with scalar 2-by-2 diagonal blocks, for which the numerical range is the convex hull of two non-concentric ellipses is given. This result is obtained by reduction to the leading special case in which $C-D^*$ also is a scalar multiple of the identity. In particular cases when in addition $\alpha-\beta$ is real or pure imaginary, the results take an especially simple form. An application to reciprocal matrices is provided.
\end{abstract}


\end{frontmatter}

\section{Introduction}

The numerical range  of a matrix $A\in\C^{n\times n}$ is defined as 
\[ W(A)=\{\scal{Ax,x}\colon x\in\C^n, \ \norm{x}=1\}. \]

We are using the standard notation $\scal{.,.}$ for the inner product on the $n$-dimensional space $\C^n$ and $\norm{.}$ for the norm associated with it: $\norm{x}^2=\scal{x,x}$. 

It is well known that $W(A)$ is a convex compact subset of $\C$ containing the spectrum $\sigma(A)$ of $A$, and thus its convex hull $\conv\sigma(A)$. For normal matrices $A$, the equality $W(A)=\conv\sigma(A)$ holds. On the other hand, for non-normal $A\in\C^{2\times 2}$ the numerical range is an elliptical disk, with the foci at the eigenvalues. These and other by now classical properties of $W(A)$ can be found, e.g., in \cite[Chapter 1]{HJ2} or more recent \cite[Chapter~6]{DGSV}.

The shape of $W(A)$ for $A\in\C^{3\times 3}$ is also known, see \cite{Ki} (or its translation \cite{Ki08} into English) for the classification and \cite{KRS} for the pertinent tests. However, for $n\geq 4$ many questions remain open. In particular, it is of interest when the boundary $\partial W(A)$ of the numerical range contains an elliptical arc \cite{GauWu081}. 

A useful tool in studying properties of $W(A)$ is the so called {\em Numerical range (NR) generating curve} $C(A)$, also introduced in \cite{Ki} --- the {\em Kippenhahn curve} in the terminology of \cite[Chapter~13]{DGSV}, where a very lucid and detailed description of $C(A)$ is given. This curve is defined uniquely by having exactly $n$ tangent lines in each direction $e^{i\theta}$, intercepting the family of orthogonal lines at the eigenvalues $\lambda_j(\theta)$ of $\im(e^{-i\theta}A)$. As it happens, $W(A)=\conv C(A)$, and so $\partial W(A)$ consists of some arcs of $C(A)$ possibly connected by line segments. 

We will denote the characteristic polynomial of $\im(e^{-i\theta}A)$ by $P_A(\lambda,\theta)$ and call it the {\em NR generating polynomial} of $A$. From the description of $C(A)$ it follows in particular that it contains an ellipse if and only if $P_A(\lambda,\theta)$ is divisible by a polynomial quadratic in $\lambda$. More specifically (as can be established by direct computations similar to those carried out say in \cite{ChienHung}), 
an ellipse centered at $p+iq$ ($p,q\in\R$) corresponds to a factor of the form 
\eq{qua} (\lambda+p\sin\theta-q\cos\theta)^2+x\cos 2\theta+y\sin 2\theta-z\en 
with some $x,y,z\in\R$ satisfying $z\geq\sqrt{x^2+y^2}$. Note that if $z=\sqrt{x^2+y^2}$, then the quadratic \eqref{qua} factors further into two linear functions in $\lambda$, and the ellipse in question degenerates into the doubleton of its foci. 

If $n=4$ and $P_A(\lambda,\theta)$ is divisible by \eqref{qua}, then the quotient is of the same type. Consequently, for $A\in\C^{4\times 4}$ the boundary of $W(A)$ contains an elliptical arc if and only if 
$C(A)$ consists of two ellipses, one of which is possibly degenerate. So, $\partial W(A)$ contains an elliptical arc if and only if $W(A)$ is an elliptical disk, the convex hull of two ellipses, or the convex hull of an ellipse and one or two points (the latter being an option only if $A$ is unitarily reducible). The respective criteria were established in \cite{Gau06}. However, these criteria are not stated in terms of $A$ directly and therefore not easy to verify. More can be done if $A$ enjoys some additional structure. 

This paper is devoted to 4-by-4 matrices of the form 
\eq{A} A=\begin{bmatrix} \alpha I & C \\ D & \beta I \end{bmatrix}.\en 
In \cite[Section~4]{GeS}, we have provided necessary and sufficient conditions for such matrices to have $W(A)$ in the shape of an elliptical disk or the convex hull of two concentric ellipses. Here we treat the remaining case, when $W(A)$ is the convex hull of two ellipses with distinct centers. For convenience of reference, this shape is called {\em bi-elliptical} in what follows.  

The special case of matrices \eqref{A} with $C-D^*$ being a scalar multiple of the identity is considered in Section~\ref{s:lsc}, after some preliminary technicalities disposed of in Section~\ref{s:pre}. Under additional restrictions on $\alpha-\beta$, this result is further simplified in Section~\ref{s:ex} which also contains several numerical examples. As it happens, the general case can be reduced to the one tackled in Section~\ref{s:lsc}; this reduction is carried over in Section~\ref{s:gen}. The final Section~\ref{s:rec} contains an application to so called reciprocal matrices.

\section{Preliminary results} \label{s:pre}

Passing from $A$ to $A-\frac{\alpha+\beta}{2}I$, we may (and will, in what follows) without loss of generality suppose that in \eqref{A} $\beta=-\alpha$.
As in \cite{GeS}, we will also use the notation \eq{HZ} H=C^*C+DD^*,\quad Z=DC. \en 

\begin{lem}\label{th:leig}Let $A$ be given by \eqref{A}, with $\beta=-\alpha$. Then:

{\em (i)}  The  eigenvalues of $A$ 
are $\pm\sigma_1, \ \pm\sigma_2$, where $\sigma_j=\sqrt{z_j+\alpha^2}$ and $z_1,z_2$ are the eigenvalues of $Z$; 

{\em (ii)} 	 The  eigenvalues of $\im(e^{-i\theta}A)$ 
are $\pm\lambda_{j}(\theta)$, where \[ \lambda_j(\theta)=\sqrt{(\im(e^{-i\theta}\alpha))^2+\mu_j(\theta)/4} \quad  (j=1,2),\] and $\mu_j(\theta)$ are the eigenvalues of 
\eq{uni} \left(e^{i\theta}C^*-e^{-i\theta}D\right)\left(e^{-i\theta}C-e^{i\theta}D^*\right)=H-2\re(e^{-2i\theta}Z);\en 
	
{\em (iii)} The NR generating polynomial $P_A$ of $A$ is $\lambda^4-\Xi_1(\theta)\lambda^2+\Xi_2(\theta)$, where 
\begin{multline} \label{xi1} 
\Xi_1(\theta)
= 
\frac{1}{4}\tr H + |\alpha|^2 - \left (\frac{1}{2} \re \tr Z + \re \alpha^2  \right) \cos 2 \theta \\ - \left( \frac{1}{2} \im \tr Z  + \im \alpha^2 \right)  \sin 2\theta, 
\end{multline}  
\begin{multline}\label{eta1} 
16\Xi_2(\theta)
	= 
	6|\alpha|^4 
	+ 
	2|\alpha|^2 \tr H + 2\re(\overline{\alpha^2} \tr Z)
	+ 
	\frac{1}{2}\tr(H)^2 - \frac{1}{2} \tr(H^2) + |\tr(Z)|^2 \\ - \tr(Z Z^*)-\re\zeta_1\cos 2\theta-\im\zeta_1\sin 2\theta+\re\zeta_2\cos 4\theta+\im\zeta_2\sin 4\theta,   
\end{multline}  with $\zeta_1, \zeta_2$ given by 
\begin{align*} \zeta_1 & = 	8|\alpha|^2\alpha^2 + 2\alpha^2\tr H 
+ 4|\alpha|^2  \tr Z  + 2 \tr Z \tr H - 2 \tr ZH, \\
\zeta_2  & = 	2\det Z + 2\alpha^2 \tr Z + 2\alpha^4. \end{align*} 
\end{lem} 
\begin{proof}For (i) and (ii), the result follows by using Schur complement formula when computing the respective characteristic polynomials; the pertinent computation for (ii) is actually contained in the proof of \cite[Lemma~2.1]{GeS}.  
Expanding $P_A(\lambda,\theta)= (\lambda^2-\lambda_1^2(\theta))(\lambda^2-\lambda_2^2(\theta))$, we derive (iii) from (ii). \end{proof} 
Observe that $P_A$ is an even function of $\lambda$. This agrees with the result of \cite{GeS} for matrices \eqref{A} with arbitrary block sizes, implying that $C(A)$ is symmetric with respect to the origin. From here we immediately obtain 

\begin{prop} \label{th:nec}Suppose $A\in\C^{4\times 4}$ of the form \eqref{A} is such that $C(A)$ consists of two non-concentric ellipses. Then, for an appropriate choice of signs of $\sigma_j$,  $\partial W(A)$ contains a pair of parallel line segments coinciding in length and direction with $\sigma_1+\sigma_2$. \end{prop} 
\begin{proof}Let $C(A)$ consist of the ellipses $E_1,E_2$. Due to the central symmetry of $C(A)$, either both $E_1,E_2$ also are symmetric with respect to the origin, or $E_1 (:= E)=-E_2\neq -E$.   The former case is excluded, because otherwise $E_1, E_2$ would be concentric. Furthermore, the foci of $\pm E$ are the eigenvalues of $A$, and so (relabeling $z_1, z_2$ if needed, and choosing the square roots signs appropriately) we may suppose that $\sigma_1,\sigma_2$ are the foci of $E$. 
	
Consider now the composition $\mathcal S$ of two symmetries, one with respect to the origin and the other with respect to the center of $E$.  By its construction, $\mathcal S$ is a shift and, since  $\mathcal S(-(\sigma_1+\sigma_2)/2)=(\sigma_1+\sigma_2)/2$, it is the shift by $\sigma_1+\sigma_2$. So, 
\[ E={\mathcal S}(-E)=-E+(\sigma_1+\sigma_2). \] The flat portions on the boundary of $W(A)$ are therefore the common tangents of $E$ and $-E$, and the endpoints of each differ by $\sigma_1+\sigma_2$.  \end{proof} 

\section{Leading special case} \label{s:lsc} 

Let \eq{AB} A= \begin{bmatrix} \alpha I & B^*+I \\ B-I & -\alpha  I\end{bmatrix}. \en
This is a particular case of \eqref{A} in which $C=B^*+I$ and $D=B-I$.
Respectively, \eqref{HZ} takes the form 
\eq{HZB} H= 2(BB^*+I), \quad Z= BB^*-I+2i \im B. \en 
Let us denote the eigenvalues of $\im B$ by $\beta_1,\beta_2$ while keeping the notation $\pm\sigma_j$ for the eigenvalues of $A$. Let us also write $\alpha$ as $u+iv$ ($u,v\in\R$). 
\begin{thm}\label{th:spec}For $A$ as in \eqref{AB}, $W(A)$ is bi-elliptical if and only if $B$ is not normal and 
\eq{spco} (1+v^2) (\sigma_1+\sigma_2)^2=(\beta_1-\beta_2)^2. \en \end{thm} 
Both in the proof of Theorem~\ref{th:spec} and its further application, putting block $B$ in
an upper triangular form \eq{utb} B=\begin{bmatrix} b_1 & b \\ 0 & b_2 \end{bmatrix}\en  and rewording conditions in terms of its entries proves to be useful. This can be done by a block diagonal unitary similarity of $A$, preserving its structure \eqref{AB}. Moreover, it can be arranged that $b\geq 0$. So, without loss of generality 
\eq{AB1} A= \begin{bmatrix} \alpha & 0 & \overline{b_1}+1 & 0 \\ 0 & \alpha & b & \overline{b_2}+1 \\
b_1-1 & b & -\alpha & 0 \\ 0 & b_2-1 & 0 & -\alpha \end{bmatrix}. \en
Let us also denote $\re b_j=\xi_j,\ \im b_j=\eta_j, \ j=1,2$.

Plugging in the values of $\sigma_j$ from Lemma~\ref{th:leig}, condition \eqref{spco} can be rewritten as  
\eq{spco1} \tr Z +2\alpha^2+2\sqrt{\det Z+\alpha^2\tr Z+\alpha^4}=4p^2, \en
where  
\eq{spr} \sqrt{\frac{(\eta_1-\eta_2)^2+b^2}{1+v^2}}:= 2p.\en Rewriting \eqref{spco1} as 
\[ 2\sqrt{\det Z+\alpha^2\tr Z+\alpha^4}=4p^2 -\tr Z-2\alpha^2\]
and taking the square, we find that \eqref{spco} is equivalent to $T=0$, where 
\eq{t1} T:=16p^4 - (8\tr Z + 16\alpha^2) p^2 + (\tr Z)^2 - 4\det Z. \en
This condition is in its turn equivalent to both $\re T$ and $\im T$ being equal to zero. For future use observe therefore that 	
\begin{align}\label{ret}
\re T = & \left(4p^2-(b^2+\abs{b_1}^2+\abs{b_2}^2)\right)^2-16u^2p^2-4\abs{b_1}^2\abs{b_2}^2,\\
\im T = & 16v\left(v(\eta_1 + \eta_2) - 2u\right) p^2+ 4 (\xi_1^2-\xi_2^2)(\eta_1 - \eta_2).\label{imt}
\end{align}
\smallskip 
{\em Proof of Theorem~\ref{th:spec}.} {\sl Necessity.} If $B$ is normal then it follows from \eqref{HZB} that so is $Z$. Moreover, $Z$ and $H$ commute. This situation falls under the setting of \cite[Theorem 4.1]{GeS}, according to which $W(A)$ is the convex hull of two ellipses, but these ellipses are concentric. So, $B$ is not normal.  

After yet another unitary (this time, permutational) similarity, the matrix \eqref{AB1} becomes  	 
\[   A_0= \begin{bmatrix} \alpha & \overline{b_1}+1 & 0 & 0 \\ b_1-1  & -\alpha &  b & 0 \\
0  & b & \alpha & \overline{b_2}+1 \\ 0 & 0 & b_2-1 & -\alpha \end{bmatrix}. \]
This matrix is tridiagonal. Moreover, $b\neq 0$ since $B$ is not normal. In terminology of \cite{BS041} it implies that $A_0$ is {\em proper} (meaning that entries in the positions $(j-1,j)$ and $(j,j-1)$ cannot simultaneously equal zero). Invoking \cite[Theorem 10]{BS041}, we conclude that the only flat portions of $\partial W(A)$ are horizontal. Their $y$-coordinates are determined by the eigenvalues of $\im A_0$, which is the direct sum of two copies of $\begin{bmatrix} v & -i \\ i & -v\end{bmatrix}$, and therefore equal $\pm\sqrt{1+v^2}$. The lengths of these portions are equal to the spread of the compression $H_0$ of $\re A_0$ onto the (2-dimensional) subspace generated by the eigenvectors of $\im A_0$ corresponding to either of its eigenvalues. Denoting $f=\sqrt{1+v^2}-v$ and $k=1/\sqrt{1+f^2}$, an orthonormal basis in one of these subspaces can be chosen as \[ k[-i,f,0,0]^T \text{ and } k[0,0,-i,f]^T. \] The matrix of $H_0$ in this basis is 
\[ \frac{1}{2\sqrt{1+v^2}}\begin{bmatrix} u(f^{-1}-f)+2\eta_1 & -ib \\ ib &  u(f^{-1}-f)+2\eta_2\end{bmatrix}, \]
and the spread of this matrix equals $2p$ defined by \eqref{spr}.
It remains to observe that $(\beta_1-\beta_2)^2=(\eta_1-\eta_2)^2+b^2$ and to make use of Proposition~\ref{th:nec}. 	

{\sl Sufficiency.} Suppose that \eqref{spco} holds.
We will now show that under this condition the NR generating polynomial of $A$  factors as 
\eq{fact} P_A(\lambda,\theta)=\left((\lambda+p\sin\theta)^2+\Omega(\theta)\right)\left((\lambda-p\sin\theta)^2+\Omega(\theta)\right),\en 
where $\Omega(\theta)=x\cos 2\theta+y\sin 2\theta-z$ for some appropriate choice of real parameters $x,y,z$.

According to Lemma~\ref{th:leig}(iii), \eqref{fact} is equivalent to 
\eq{eq1} 2\left(p^2\sin^2\theta-\Omega(\theta)\right)=\Xi_1(\theta), \quad \left(p^2\sin^2\theta+\Omega(\theta)\right)^2=\Xi_2(\theta),\en 
where $\Xi_1$ and $\Xi_2$ are given by \eqref{xi1} and \eqref{eta1}, respectively. 
The first equality in \eqref{eq1} defines $x,y,z$ uniquely as 
\eq{xyz}
\begin{cases}
x = \frac{1}{4} \re \tr Z + \frac{1}{2} \re(\alpha^2) - \frac{p^2}{2},
\\
y = \frac{1}{4} \im \tr Z + \frac{1}{2} \im (\alpha^2),
\\ 
z = 
\frac{1}{8} \tr H +  \frac{|\alpha|^2}{2} - \frac{p^2}{2}, 
\end{cases}\en 
or, in terms of $B$ directly:
\eq{xyz1} \begin{cases}
x = 
\frac{1}{4} (b^2 + |b_1|^2 + |b_2|^2 - 2) + \frac{1}{2} \re (\alpha^2) - \frac{1}{8}\frac{b^2 + (\eta_1 - \eta_2)^2}{1+v^2}, \\
y = 
\frac{1}{2} (\eta_1 + \eta_2) + \frac{1}{2} \im (\alpha^2), \\ 
z = 
\frac{1}{4} (b^2 + |b_1|^2 + |b_2|^2 + 2 ) + \frac{1}{2} |\alpha|^2 - \frac{1}{8}\frac{b^2 + (\eta_1 - \eta_2)^2}{1+v^2}.
\end{cases}\en 
Incidentally, with this choice of $x,y,z$, 	
the second equality in \eqref{eq1} also holds. Here is the pertinent chain of computations. First, form \eqref{xyz}:
	\begin{align*}
	4\left[ p^2 \sin^2(\theta) + \Omega (\theta)\right] 
	& = 
	- \frac{1}{2} \tr H - 2 |\alpha|^2 + 4p^2
	\\ & +
	\left[ 
	\re \tr Z + 2 \re(\alpha^2)
	- 4p^2 
	\right] \cos 2 \theta 
	\\ & + 
	\left[ 
	\im \tr Z + 2 \im(\alpha^2)
	\right] \sin 2 \theta 
	\end{align*}
	Squaring we obtain:
	\begin{align*}
	& 16\left[ p^2 \sin^2(\theta) + \Omega(\theta) \right]^2
	\\ & =  
	24p^4 - 4p^2 \left( \re ( \tr Z + 2\alpha^2) + \tr H + 4|\alpha|^2 \right) 
	\\ & 
	+ \frac{1}{2}  | \tr Z + 2 \alpha^2|^2 + \left(\frac{1}{2} \tr H +  2|\alpha|^2 \right)^2 
	\\ & - 
	8\re \left[ \left(2p^2 - \frac{1}{2} \tr Z - \alpha^2 \right)\left(2p^2 - \frac{1}{4} \tr H -  |\alpha|^2\right) \right] \cos 2 \theta 
	\\ & - 
	8\im \left[ \left(2p^2 - \frac{1}{2} \tr Z - \alpha^2 \right)\left(2p^2 - \frac{1}{4} \tr H -  |\alpha|^2\right) \right] \sin 2 \theta 
	\\ & + 
	2\re \left[ (2p^2 - \frac{1}{2} \tr Z - \alpha^2)^2 \right] \cos 4 \theta 
	\\ & +
	2\im \left[ (2p^2 - \frac{1}{2} \tr Z - \alpha^2)^2 \right] \sin 4 \theta .
	\end{align*}
Plugging in $\Xi_2$ from \eqref{eta1}:
\begin{align} \label{pomegasq}
& 32\left[ \left(p^2\sin^2\theta+\Omega(\theta)\right)^2-\Xi_2 (\theta) \right]
\\ & = \nonumber
3\bigg[16p^4 - \frac{8}{3} \left( \re ( \tr Z + 2\alpha^2) + \tr H + 4|\alpha|^2 \right) p^2
\\ & \nonumber
+ \frac{2}{3}\tr(Z^*Z) - \frac{1}{3}|\tr Z|^2  + \frac{1}{3} \tr(H^2) - \frac{1}{6}\tr(H)^2 \bigg]
\\ & - \nonumber
4\re \left[16p^4 - 2(\tr H + 2\tr Z +  4\alpha^2 + 4|\alpha|^2)p^2 + \tr ZH - \frac{1}{2}\tr Z \tr H 
\right] \cos 2 \theta 
\\ & - \nonumber
2\im \left[-8(\tr Z +  2\alpha^2)p^2 + 2\tr ZH - \tr Z \tr H 
\right] \sin 2 \theta 
\\ & + \nonumber
\re \left[
16p^4 - 8(\tr Z + 2\alpha^2) p^2 + (\tr Z)^2 - 4\det(Z)
\right] \cos 4 \theta 
\\ & + \nonumber
\im \left[
-8(\tr Z + 2\alpha^2) p^2 + (\tr Z)^2 - 4\det(Z)
\right] \sin 4 \theta. 
\end{align}
Rewriting the right hand side of \eqref{pomegasq} in terms of the entries of $B$ we obtain, with the use of \eqref{ret}--\eqref{imt}:
\begin{align*}
    &  \left(p^2\sin^2\theta+\Omega(\theta)\right)^2-\Xi_2(\theta) 
    \\ &=   
    \frac{1}{32} \left[ 3 \re T - 4\re T \cos 2 \theta - 2\im T \sin 2 \theta + \re T \cos 4 \theta + \im T \sin 4 \theta \right].
\end{align*}
So, condition \eqref{spco}, equivalent to $T=0$, indeed implies the second equality in \eqref{eq1}.

Finally, from \eqref{xyz1}:  $z = x + (1 + v^2)$, and furthermore,
\begin{multline*}
z^2 - x^2 - y^2 = (1+v^2)^2 + 2(1+v^2)x - y^2 \\  =
\frac{(1+v^2)(\xi_1^2 + \xi_2^2)}{2} +
\frac{(1+2v^2)b^2}{4}
+ \left[u - \frac{(\eta_1+\eta_2) v}{2}\right]^2 +
\frac{v^2(\eta_1 - \eta_2)^2}{4}>0
\end{multline*}
since $b\neq 0$. Factorization \eqref{fact} therefore generates $C(A)$ consisting of two non-degenerate ellipses. \qed

\section{Follow up observations and examples} \label{s:ex}

Suppose that in \eqref{AB} the parameter $\alpha$ is real or pure imaginary. Criterion established in Section~\ref{s:lsc} can then be recast explicitly in terms of $B$. This is done in two theorems below, stated for $B$ as in \eqref{utb}. Note however that the results can be easily reworded without putting $B$ in a triangular form. Indeed, $\{b_1,b_2\}$ is the spectrum of $B$, while $b^2=\norm{B}_F-\abs{b_1}^2-\abs{b_2}^2$, with $\norm{.}_F$ denoting the Frobenius norm.

\begin{thm}\label{th:real}Let $\alpha (=u)\in\R$. Then the numerical range of matrix \eqref{AB1} is bi-elliptical if and only if $b\neq 0$  and either \\ {\em (i)} $\eta_1=\eta_2$ and $4b^2u^2=\left(\xi_1^2-\xi_2^2\right)^2$, or {\em (ii)} $u=\xi_1=\xi_2=0$.  \end{thm}
\begin{proof} Since $v=0$, according to \eqref{imt} $\im T=0$ if and only if $\eta_1=\eta_2$ or $\xi_1^2=\xi_2^2$. On the other hand, $4p^2=(\eta_1-\eta_2)^2+b^2$ due to \eqref{spr}, and so \eqref{ret} 
takes the form \eq{ret1} \left( (\eta_1-\eta_2)^2-(\abs{b_1}^2+\abs{b_2}^2)\right)^2- 4u^2\left( (\eta_1-\eta_2)^2+b^2\right)-4\abs{b_1}^2\abs{b_2}^2. \en 
If $\eta_1=\eta_2$, then \eqref{ret1} vanishing amounts to
$\left(\abs{b_1}^2-\abs{b_2}^2\right)^2 = 4u^2 b^2$, which is exactly case (i).  If $\eta_1\neq\eta_2$, denote the coinciding values of $\xi_1^2$ and $\xi_2^2$ by $\xi^2$. Expression \eqref{ret1} then simplifies to \[-4\xi^2(\eta_1-\eta_2)^2-4u^2\left( (\eta_1-\eta_2)^2+b^2\right),\] and so it vanishes if and only if $u=\xi=0$. This is case (ii). 
\end{proof} 

\begin{cor}\label{th:za}Let $A$ be of the form \eqref{AB1} with zero main diagonal. Then $W(A)$ is bi-elliptical if and only if $b\neq 0$ and either \\ {\em (i)}  $\im b_1=\im b_2$ and  $\abs{b_1}=\abs{b_2}$, or {\em (ii)} $\re b_1=\re b_2=0$. 
\end{cor} 
Note that in \cite[Theorem 3.14]{Yeh} the result of Corollary~\ref{th:za} was established under additional restrictions $b=\sqrt{(1-\abs{b_1})^2(1-\abs{b_2})^2}$, $\abs{b_1},\abs{b_2}<1$. 

According to Corollary~\ref{th:za}, for $\alpha=0$ the value of $b$ is irrelevant (as long as it is different from zero, of course). The situation is quite the opposite for non-zero values of $\alpha$.
\begin{cor}\label{th:ub}Let $A$ be of the form \eqref{AB1} with $\alpha\neq 0$. Then there is at most one value of $b (>0)$ for which $W(A)$ is bi-elliptical. \end{cor}
\begin{proof}For $\alpha\in\R$ the result follows from Theorem~\ref{th:real}, with an explicit expression for $b$. For non-real values of $\alpha$, observe that condition $\re T=0$ with the use of \eqref{spr} and \eqref{ret} yields a quadratic equation for $b^2$ with one non-positive root, and thus a unique positive one. 
\end{proof}
It is possible to derive the explicit value of $b$ for non-real $\alpha$ as well, but the expression is somewhat cumbersome. We will restrict our attention to another special case, when $\alpha$ is pure imaginary.
\begin{thm}\label{th:im}Let $\alpha$ be pure imaginary: $\alpha=iv\neq 0$. Then the numerical range of matrix \eqref{AB1} is bi-elliptical if and only if $b\neq 0$,\eq{recri} \abs{b_1}=\abs{b_2}, \text{ and } v^2b^2=(\eta_1-\eta_2)^2.\en  \end{thm}
\begin{proof} Setting $u=0$ in \eqref{ret} and \eqref{imt}, we see that conditions $\re T=0$ and $\im T=0$ simplify respectively to 
\eq{pb} 4p^2-b^2=\left(\abs{b_1}\pm\abs{b_2}\right)^2 \en	
and 	
\eq{pv} 4v^2p^2(\eta_1 + \eta_2)+ (\xi_1^2-\xi_2^2)(\eta_1 - \eta_2)=0.\en	
Moreover, due to \eqref{spr} and \eqref{pb}: 
\[ 4v^2p^2=(\eta_1-\eta_2)^2-(4p^2-b^2)=(\eta_1-\eta_2)^2-\left(\abs{b_1}\pm\abs{b_2}\right)^2. \]
Plugging this expression for $4v^2p^2$ into \eqref{pv}, after some additonal simplifications we arrive at 
\eq{pro} \left(\abs{b_1}\pm\abs{b_2}\right)\left(\eta_1\abs{b_2}\pm\eta_2\abs{b_1}\right)=0, \en
with the signs matching that in \eqref{pb}. 
Note that the system of conditions \eqref{pb},\eqref{pv} is equivalent to \eqref{pb},\eqref{pro}, with $p$ given by \eqref{spr}. 

The sufficiency of \eqref{recri} for \eqref{spco} to hold is now immediate. To prove the necessity, 
observe that from \eqref{spr}, since $b\neq 0$:
\[ 4p^2-b^2=\frac{(\eta_1-\eta_2)^2-v^2b^2}{1+v^2}<(\eta_1-\eta_2)^2
\leq (\abs{\eta_1}+\abs{\eta_2})^2\leq \left(\abs{b_1}+\abs{b_2}\right)^2, \] 	
and therefore \eqref{pb} holds with the lower sign. Then so does \eqref{pro}. 

Suppose now that $\abs{b_1}\neq\abs{b_2}$. Condition \eqref{pro} then boils down to $\eta_1\abs{b_2}=\eta_2\abs{b_1}$, which in its turn implies that $\eta_1,\eta_2$ are of the same sign, and $\arg b_1=\arg b_2$ or $\arg b_1=\pi-\arg b_2$. The above inequality for $4p^2-b^2$ can therefore be strengthened as follows:  
\[ 4p^2-b^2=\frac{(\eta_1-\eta_2)^2-v^2b^2}{1+v^2}<(\eta_1-\eta_2)^2
=(\abs{\eta_1}-\abs{\eta_2})^2\leq \left(\abs{b_1}-\abs{b_2}\right)^2. \] 		
This is a contradiction with \eqref{pb}. 

So, in fact $\abs{b_1}=\abs{b_2}$, and \eqref{pb} (with the correct choice of the sign) implies $4p^2-b^2=0$,  thus leading to \eqref{recri}.  \end{proof}

\begin{figure}[H]
    \centering
    \includegraphics[scale=0.50]{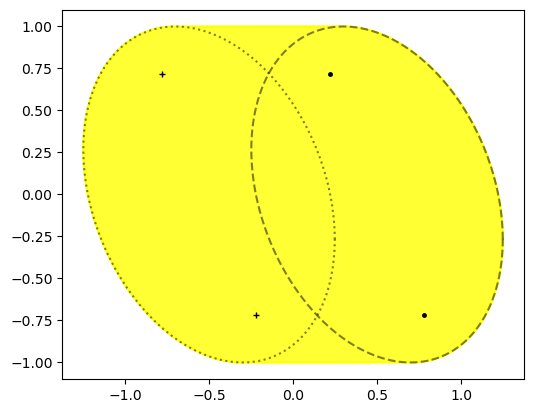}
    \includegraphics[scale=0.50]{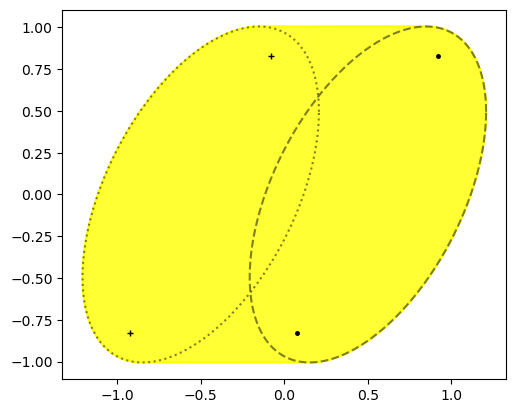}
    \label{fig:u0andv0}
        \caption
        {
        \textbf{On the left -} an example illustrating theorem \ref{th:real} with 
        $\alpha = \frac{1}{10}$, $b=1$, 
        $b_1 = \frac{1}{5}\left(3 - i \right)$  
        and
        $b_2 = \frac{1}{5}\left(2 - i \right)$  
        .
        \textbf{On the right -} an example illustrating theorem \ref{th:im} with 
        $\alpha = \frac{1}{10}i$, $b=1$, 
        $b_1 = \frac{1}{10}\left(3 + 4i \right)$  
        and 
        $b_2 = \frac{1}{10}\left(4 + 3i \right)$ 
        .
        }
\end{figure}

\section{General case} \label{s:gen}
We now return to arbitrary matrices \eqref{A}, not supposing a priori the special form \eqref{AB}. 

\begin{thm}\label{th:gen} Let $A$ be of the form \eqref{A}, with the usual convention $\beta=-\alpha$. Then $W(A)$ is bi-elliptical if and only if \\ {\em (i)} there exists $\theta$ for which $e^{-i\theta}C-e^{i\theta}D^*$ is a scalar multiple of a unitary matrix while  $(e^{-i\theta}C-e^{i\theta}D^*)D$ is not normal, and\\ {\em (ii)} for this value of $\theta$, 

\eq{gen} 4\sqrt{\det\left(\im(e^{-i\theta}A) \right)}(\sigma_1(\theta)+\sigma_2(\theta))^2=(s_1(\theta)-s_2(\theta))^2. \en 
\end{thm}
Here $\pm\sigma_1(\theta),\pm\sigma_2(\theta)$ are the eigenvalues of $e^{-i\theta}A$ (so that $\sigma_j(0)=\sigma_j$, $j=1,2$) and $s_1(\theta),s_2(\theta)$ are the non-repeating eigenvalues of $\im\left((e^{-i\theta}A)^2\right)$.  

It will become clear from the proof of Theorem~\ref{th:gen} that $\theta$ satisfying condition (i) is unique $\hspace{-.25cm}\mod\pi$. Also, if (i) holds then the eigenvalues $\pm\lambda_1(\theta),\pm\lambda_2(\theta)$ of $\im(e^{-i\theta}A)$ actually satisfy $\lambda_1(\theta)=\lambda_2(\theta)$ and, if (ii) also holds, then $\theta=\arg(\sigma_1+\sigma_2)\mod\pi$. So, with an appropriate choice of the signs condition \eqref{gen} can be rewritten as \[ \left(\lambda_1(\theta)+\lambda_2(\theta)\right)\left(\sigma_1(\theta)+\sigma_2(\theta)\right)
=s_1(\theta)-s_2(\theta). \]
Finally, condition $s_1(\theta)\neq s_2(\theta)$ follows from 
$(e^{-i\theta}C-e^{i\theta}D^*)D$ not being normal. 

\begin{proof} To establish necessity of (i), choose $e^{i\theta}$ as the direction of the line segments on $\partial W(A)$, as guaranteed by Proposition~\ref{th:nec}. The matrix $\im(e^{-i\theta}A)$ then must have repeated eigenvalues. According to Lemma~\ref{th:leig}(ii), this happens if and only if the 2-by-2 hermitian matrix defined by \eqref{uni} has coinciding eigenvalues, and is therefore a scalar multiple of the identity. 
This, in turn, is equivalent to $e^{-i\theta}C-e^{i\theta}D^*$ being a scalar multiple of a unitary matrix.

This scalar multiple cannot be zero, since otherwise $Z=e^{2i\theta}C^*C$ while $H=2C^*C$, and so $Z$ is a normal matrix commuting with $H$. As was already mentioned in the proof of Theorem~\ref{th:spec}, this situation corresponds to $W(A)$ being the convex hull of two concentric ellipses, and thus leads to a contradiction. 

Both condition \eqref{gen} and the shape of $W(A)$ persist under scaling of $A$. So, we may for the rest  of the proof suppose that $\theta=0$ and $C^*-D=2W$ where $W$ is unitary. Furthermore, 
a unitary similarity of $A$ via $\diag[I,W]$ leaves condition \eqref{gen} invariant, while replacing  $C$ by $C_1=CW$ and $D$ by $D_1=W^*D$. Since $C_1^*-D_1= W^*(C^*-D)=2I$, it suffices to consider the case $W=I$. 

This brings us into the setting of Theorem~\ref{th:spec}, with $B= D_1+I (= C_1^*-I)$. Since this $B$ is (or is not) normal simultaneously with $(e^{-i\theta}C-e^{i\theta}D^*)D$, condition (i) is indeed necessary. 

Supposing that (i) holds, it only remains to show that \eqref{gen} with $\theta=0$ is \eqref{spco} in disguise. But indeed, for $A$ as in \eqref{AB}:
\[ 1+v^2=\sqrt{\det(\im A)}, \] while $\im (A^2)$ is the direct sum of two copies of $\im(\alpha^2)I+2\im B$, and so 
$s_1(0)-s_2(0)=2(\beta_1-\beta_2)$. 
\end{proof} 
\noindent \textbf{Example.}
	Let
	\[
	A = \begin{bmatrix} \label{eq:exgenA}
		4 & 0 & 4-8i & 0 \\
		0 & 4 &  5-5i & 4-12i \\
		-2+6i & 5-5i & -4 & 0 \\
		0 & 2+6i & 0 & -4
	\end{bmatrix}.
	\]
The matrix $e^{-i\theta}C-e^{i\theta}D^*$ is upper triangular, and so it can be normal only if its lower left entry is also zero. This requirement is equivalent to $\theta=-\pi/4\mod\pi$. On the other hand, with this choice of $\theta$ indeed $e^{-i\theta}C-e^{i\theta}D^*=\pm 10\sqrt{2}I$ is a scalar multiple of a unitary matrix (in this case, even the identity). Moreover, since $D$ is not normal, the product $(e^{-i\theta}C-e^{i\theta}D^*)D$ is not normal either. So, condition (i) of Theorem~\ref{th:gen} holds.

Furthermore, since 
\begin{align*}
	 e^{i\pi/4} A = 
	\sqrt{2} 
	\begin{bmatrix}
		2+ 2i & 0 & 6-2i & 0 \\
		0 & 2+ 2i & 5 & 8-4 i \\
		- 4+2i & 5 & -2- 2i & 0 \\
		0 & -2+4 i & 0 & -2- 2i
	\end{bmatrix}, 
\end{align*} somewhat lengthy but direct computations yield: $\sigma_1(-\pi/4)+\sigma_2(-\pi/4) = 5\sqrt{2}$, $s_1(-\pi/4) - s_2(-\pi/4) = 20\sqrt{29}$, and $\sqrt{\det\left( (\im(e^{i\pi/4}A)\right)}= 58$. 
So, condition (ii) of Theorem~\ref{th:gen} also holds, and $W(A)$ is bi-elliptical.

\begin{figure}[H]
	    \centering
	    \includegraphics[scale=0.6]{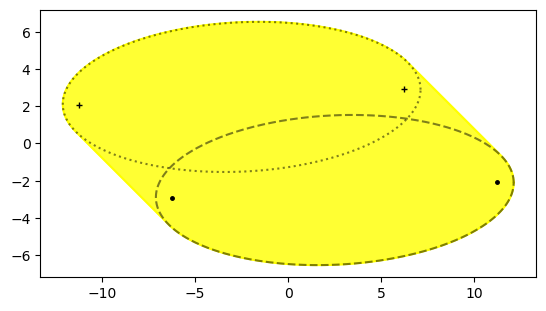}
	    \label{fig:gen}
	    \caption{The numerical range of $A$ from the Example.
	    }
	\end{figure}

The bi-elliptical shape of $W(A)$ in the setting of Theorem~\ref{th:gen} means that in fact $W(A)=W(A_1\oplus A_2)$, where 
$A_1=-A_2\in\C^{2\times 2}$, $\sigma(A_1)=\{\sigma_1,\sigma_2\}$, and $\norm{A_1}_F^2=\norm{A}_F^2/2$. The following observation is therefore non-trivial and thus of some interest. 
\begin{prop}\label{th:irred}Any matrix \eqref{A} satisfying condition {\em (i)} of Theorem~\ref{th:gen} is unitarily irreducible. \end{prop}
In what follows, denote by $\mathcal L_\pm$ the span of $e_1,e_2$ (resp., $e_3,e_4$) --- the first/last two standard basis vectors of $\C^4$. The following auxiliary statement will be used repeatedly. 
\begin{lem}\label{th:aux}Let a reducing subspace $\mathcal L$ of the matrix \eqref{AB} have a non-trivial intersection with $\mathcal L_+$ or  $\mathcal L_-$. Then $\mathcal L=\C^4$. \end{lem} 
\begin{proof} For any $x\in\C^2$
	\[  A\begin{bmatrix}x\\ 0\end{bmatrix} =\begin{bmatrix}\alpha x\\ Bx-x\end{bmatrix}, \quad  A^*\begin{bmatrix}x\\ 0\end{bmatrix} =\begin{bmatrix}\overline{\alpha} x\\ Bx+x\end{bmatrix}\] and  
	\[ A\begin{bmatrix}0\\ x\end{bmatrix} =\begin{bmatrix}B^*x+x\\ -\alpha x\end{bmatrix},\quad  A^*\begin{bmatrix}0\\ x\end{bmatrix} =\begin{bmatrix}B^*x-x\\ -\overline{\alpha}x\end{bmatrix}.\]
So, if for some $x\in\C^2$ the subspace $\mathcal L$ contains one of the vectors $\begin{bmatrix}x\\ 0\end{bmatrix}$ or $\begin{bmatrix}0\\ x\end{bmatrix}$, it also  contains the other one, as well as $\begin{bmatrix}B^*x\\ 0\end{bmatrix}$ and  $\begin{bmatrix}0\\ Bx\end{bmatrix}$. Applying this observation to $Bx$ and $B^*x$ in place of $x$, we conclude that $\begin{bmatrix}0\\ B^*x\end{bmatrix}$, $\begin{bmatrix}Bx\\ 0\end{bmatrix}$ also lie in $\mathcal L$, and thus 
\[ \mathcal L\supseteq\Span\left\{\begin{bmatrix}x \\ 0\end{bmatrix}, \begin{bmatrix}Bx \\ 0\end{bmatrix},\begin{bmatrix}B^*x \\ 0\end{bmatrix},
\begin{bmatrix}0\\ x\end{bmatrix}, \begin{bmatrix}0 \\ Bx\end{bmatrix},\begin{bmatrix}0 \\ B^*x\end{bmatrix}
\right\}.\] Since $B$ is not normal, if such a non-zero $x$ exists, it cannot be an eigenvector of both $B$ and $B^*$.
Consequently, $\Rk\{x,Bx,B^*x\}=2$, and $\mathcal L=\C^{4\times 4}$.
\end{proof}  
{\em Proof of Proposition~\ref{th:irred}.} As it was shown in the proof of Theorem~\ref{th:gen}, condition (i) implies that by scaling, rotating, and unitary similarities the matrix $A$ can be put in form \eqref{AB} with a non-normal $B$. Since these transformations preserve unitary (ir)reducibility, we only need to consider this special case. In particular, Lemma~\ref{th:aux} is applicable.
	
Observe that any reducing subspace $\mathcal L$ of $A$ has to be invariant under 
\[ A^2=\alpha^2I+\diag[CD, DC]=(\alpha^2-1)I+\diag[B^*B+B-B^*,BB^*+B-B^*], \]
and therefore under $\im A^2$ and $\re A^2$. This implies the invariance under 
$\diag[B^*B,BB^*]$ and $\diag[\im B,\im B]$. Equivalently, $\mathcal L$ is invariant under 
$H_1=\diag[P_1,P_2]$ and $H_2=\diag[P,P]$, where $P_1,P_2$ and $P$ are spectral projections of $B^*B,BB^*$ and $\im B$, respectively. Note that these projections all have rank one, and $P_1$ does not commute with $P_2$, due to non-normality of $B$. 

With an appropriate choice of a unitary matrix $W\in\C^{2\times 2}$, a unitary similarity via $\diag[W,W]$ can be used to put $P$ in the form $\diag[1,0]$, without changing the structure \eqref{AB} of $A$. We will use the notation $\begin{bmatrix}t_j & \overline{\omega_j} \\ \omega_j & 1-t_j\end{bmatrix}$ for the resulting form of $P_j$, $j=1,2$. Here $0\leq t_j\leq 1$, with $t_j$ not equal zero (or one) simultaneously, and $\abs{\omega_j}^2=t_j(1-t_j)$. 

Since $H_2=\diag[1,0,1,0]$, a non-zero $\mathcal L$ has to contain a non-zero vector $x$ of the form $[m,0,n,0]^T$
or $[0,m,0,n]^T$. If only one of $m,n$ is non-zero, Lemma~\ref{th:aux} implies that $\mathcal L=\C^4$. So, we need now only to consider $m,n\neq 0$. The two options for the location of the non-zero entries can be treated similarly, and for the sake of definiteness we will assume that $x=[m,0,n,0]^T$.   

Applying $H_1$ and then $H_2$ we see that 
\eq{t12} [t_1m,\omega_1m,t_2n,\omega_2n]^T, \ [t_1m,0,t_2n,0]^T\in\mathcal L. \en
If $t_1\neq t_2$, comparing the second vector from \eqref{t12} with $x$ we conclude that $e_1$ or $e_3$ lies in $\mathcal L$. We can thus invoke Lemma~\ref{th:aux} again. 

It remains to consider the case $t_1=t_2$. The second inclusion in \eqref{t12} is then redundant while the first can be simplified to $y=[0,\omega_1m,0,\omega_2n]^T\in\mathcal L$. 

Since $\mathcal L$ is also invariant under $\im A=\begin{bmatrix}vI & -iI \\ iI & -vI\end{bmatrix}$, along with $x,y$ it will contain $\tilde{x}=[vm-in,0,im-vn,0]^T$ and $\tilde{y}=[0,v\omega_1m-i\omega_2n,0,i\omega_1m-v\omega_2n]^T$. 
So, $\dim\mathcal L\geq \Rk\{x,\tilde{x},y,\tilde{y}\}=\Rk \Delta_1+\Rk\Delta_2$, where \[ \Delta_1=\begin{bmatrix}m & vm-in \\ n & im-vn\end{bmatrix}, \quad \Delta_2=\begin{bmatrix}\omega_1m & \omega_1vm-i\omega_2n \\ \omega_2n & i\omega_1m-\omega_2vn\end{bmatrix}.\]  But 
$\det\Delta_1=i(m^2+n^2)-2vmn$  and $\det\Delta_2=i(\omega_1^2m^2+\omega_2^2n^2)-2\omega_1\omega_2vmn$ cannot both equal zero unless $\omega_1=\omega_2$. This, however, is precluded by non-commutativity of $P_1$ with $P_2$.
So, $\mathcal L$ is at least 3-dimensional, it therefore has a non-trivial intersection with 2-dimensional $\mathcal L_\pm$, and yet another application of Lemma~\ref{th:aux} completes the proof. \qed

\section{Reciprocal matrices} \label{s:rec} 
To illustrate the applicability of Theorem~\ref{th:gen}, let us consider a so called {\em reciprocal} 4-by-4 matrix. By definition this is a tridiagonal matrix with the off-diagonal pairs of mutually inverse entries: 

\eq{rec} A=\begin{bmatrix} 0 & a_1 & 0 & 0\\ a_1^{-1} & 0 & a_2 & 0 \\ 0 & a_2^{-1} & 0 & a_3 \\ 0 & 0 & a_3^{-1} & 0 \end{bmatrix}. \en 
A diagonal unitary similarity can be used to change the argumets of $a_j$ independently, without any effect on $W(A)$. So, without loss of generality it suffices to consider matrices \eqref{rec} with $a_j>0$ for all $j=1,2,3$. 

Formally speaking, reciprocal matrices are not of the type considered in this paper. However, by a transpositional similarity \eqref{rec} can be put in the form \eqref{A} with $\alpha=\beta=0$ and 
\eq{CDr} C= \begin{bmatrix}a_1 & 0\\ a_2^{-1} & a_3\end{bmatrix}, \quad D=\begin{bmatrix}a_1^{-1} & a_2 \\ 0 & a_3^{-1}\end{bmatrix}.\en
This observation, combined with \cite[Lemma~2.2]{GeS}, was used in K.~Vazquez's Capstone project (under the supervision of one of the authors) to obtain the criterion for the matrix \eqref{rec} to have an elliptical numerical range. A similar result was obtained simultaneously and independently by N.~Bebiano and J.~da~Provid\'encia (private communication). Stated in terms of 
\eq{Aj} A_j= \frac{a_j^2+a_j^{-2}}{2}\ (\geq 1), \en this criterion looks as follows:

\[ A_2=\frac{1+\sqrt{5}}{2}A_1+\frac{1-\sqrt{5}}{2}A_3 \text{ or }  A_2=\frac{1+\sqrt{5}}{2}A_3+\frac{1-\sqrt{5}}{2}A_1,  \]
and at  least one of the inequalities in \eqref{Aj} is strict. 

Here is what follows for reciprocal matrices \eqref{rec} by applying Theorem~\ref{th:gen} to
$A$ given by \eqref{A},\eqref{CDr}  with $\alpha=\beta=0$. 
\begin{thm}\label{th:rec}Let $A$ be as in \eqref{rec}. Then $W(A)$ is bi-elliptical if and only if $A_1=A_3>1$ (equivalently: $a_1=a_3\neq 1$ or $a_1=a_3^{-1}\neq 1$) and $A_2=1$ (equivalently: $a_2=1$). \end{thm}
\begin{proof}With $C$ and $D$ given by \eqref{CDr}, the matrix $Z$ from \eqref{HZ} is simply 
\[ Z=\begin{bmatrix} 2 & a_2a_3\\ a_2^{-1}a_3^{-1} & 1\end{bmatrix}. \]
So, the eigenvalues $z_1,z_2$ of $Z$ are $(3\pm\sqrt{5})/2$, and by Lemma~\ref{th:leig}(i) with $\alpha=0$ the spectrum of $A$ consists of the four point $\pm(1\pm\sqrt{5})/2$.

The only candidates for $\theta$ in the statement of Theorem~\ref{th:gen} in our setting are therefore integer multiples of $\pi$.  

With this choice of $\theta$, the matrix $e^{-i\theta}C-e^{i\theta}D^*$ up to the sign equals 
\[ C-D^*=\begin{bmatrix}a_1-a_1^{-1} & 0 \\ a_2^{-1}-a_2 & a_3-a_3^{-1}\end{bmatrix}. \]
It is a scalar multiple of a unitary  if and only if $ a_2^{-1}-a_2=0$ and $a_1-a_1^{-1}=\pm(a_3-a_3^{-1})$. These are exactly the conditions $A_2=1$, $A_1=A_3$. 

Finally, with $a_2=1$ the matrix 
\[ ( C-D^*)D= \begin{bmatrix} 1-a_1^{-2} & a_1-a_1^{-1} \\ 0 &  1-a_3^{-2}\end{bmatrix}  \] 
is normal if and only if $a_1=1$ (equivalently: $A_1=1$). This concludes the proof of necessity. 

To prove sufficiency, we just need to show that \eqref{gen} holds with $\theta=0$. To this end, observe that 
\[ \det(\im A)=\frac{1}{16}\abs{\det(C-D^*)}^2=\frac{1}{16}(a_1-a_1^{-1})^2(a_3-a_3^{-1})^2=\frac{1}{4}(A_1-1)^2, \]
and $\im(A^2)$ is unitarily similar to $ \im Z \oplus \im Z$ so that $s_{1,2}=\pm (a_3-a_3^{-1})/2$. 
Consequently, \[ (s_1-s_2)^2=(a_3-a_3^{-1})^2=2A_3-2=2A_1-2.\]
We see that \eqref{gen} indeed holds if we pick $\sigma_{1,2}\in\sigma(A)$
as say $(1\pm\sqrt{5})/2$.
\end{proof}  
\begin{figure}[H]
    \centering
    \includegraphics[scale=0.6]{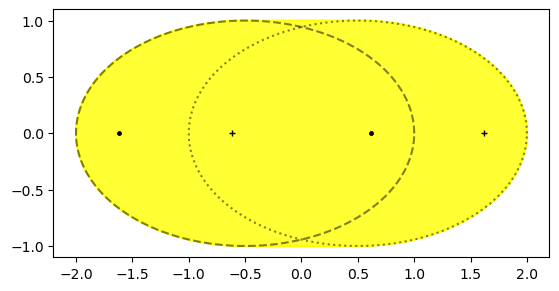}
    \label{fig:rec}
    \caption{An example illustrating Theorem~\ref{th:rec} with $A_2 = 1$ and $A_1 = A_3 = 3$.
    }
\end{figure}

\providecommand{\bysame}{\leavevmode\hbox to3em{\hrulefill}\thinspace}
\providecommand{\MR}{\relax\ifhmode\unskip\space\fi MR }
\providecommand{\MRhref}[2]{%
	\href{http://www.ams.org/mathscinet-getitem?mr=#1}{#2}
}
\providecommand{\href}[2]{#2}

\end{document}